\documentclass[11pt, english, draft]{article}
\usepackage{amssymb,amsmath}
\title{\bf Some Quantitative Characterizations of Certain Symplectic Groups}
\author{ {\bf M. Akbari} and {\bf A. R. Moghaddamfar}\\[0.2cm]
{\em Department of Mathematics, K. N. Toosi
University of Technology,}\\
 {\em P. O. Box $16315$-$1618$, Tehran, Iran}\\[0.1cm]
{\em E-mails}: {\tt moghadam@kntu.ac.ir} and {\tt
moghadam@ipm.ir}}

\newenvironment{proof}{\noindent {\em {Proof}}.}{$\square$
\medskip}
\newtheorem{definition}{Definition}
\newtheorem{corollary}{Corollary}

\newtheorem{theorem}{Theorem}

\newtheorem{lm}{Lemma}

\newtheorem{problem}{Problem}
\newtheorem{step}{Step}

\begin{document}
\maketitle
\begin{abstract}
\noindent Given a finite group $G$, denote by ${\rm D}(G)$ the
degree pattern of $G$ and by ${\rm OC}(G)$ the set of all order
components of $G$. Denote by $h_{{\rm OD}}(G)$ (resp. $h_{{\rm
OC}}(G)$) the number of isomorphism classes of finite groups $H$
satisfying conditions $|H|=|G|$ and ${\rm D}(H)={\rm D}(G)$
(resp. ${\rm OC}(H)={\rm OC}(G)$). A finite group $G$ is called
OD-characterizable (resp. OC-characterizable) if $h_{\rm
OD}(G)=1$ (resp. $h_{\rm OC}(G)=1$). Let $C=C_p(2)$ be a
symplectic group over binary field, for which $2^p-1>7$ is a
Mersenne prime. The aim of this article is to
prove that $h_{\rm OD}(C)=1=h_{\rm OC}(C)$.\\[0.3cm]
{\bf Keywords}: spectrum of a group, prime graph, degree pattern,
order component, symplectic group $C_n(q)$,
OD(OC)-characterizability of a finite group.
\end{abstract}
\def\thefootnote{ \ }
\footnotetext{{\em $2000$ Mathematics Subject Classification}:
20D05, 20D06, 20D08.}
\thanks{}
\section{Introduction}
Only finite groups will be considered. Let $G$ be a group,
$\pi(G)$ the set of all prime divisors of its order and
$\omega(G)$ be the spectrum of $G$, that is the set of its
element orders. The {\em prime graph} ${\rm GK}(G)$ (or {\em
Gruenberg-Kegel graph}) of $G$ is a simple graph whose vertex set
is $\pi(G)$ and two distinct vertices $p$ and $q$ are joined by
an edge (and we write $p\sim q$) if and only if $pq\in
\omega(G)$. Let $t(G)$ be the number of connected components of
${\rm GK}(G)$. The $i$th connected component is denoted by
$\pi_i(G)$ for each $i=1, 2, \ldots, t(G)$. In the case when
$2\in \pi(G)$, we assume that $2\in \pi_1(G)$.

The classification of simple groups with disconnected prime graph
was obtained by Williams \cite{Wil} and Kondrat\'ev \cite{k}.
Moreover, a corrected list of these groups can be found in
\cite{km}. Recall that a {\em clique} in a graph is a set of
pairwise adjacent vertices. Note that for all non-abelian simple
groups $S$ with disconnected prime graph, all connected components
$\pi_i(S)$ for $2\leq i\leq t(S)$ are clique, for instance, see
\cite{k}, \cite{suz} and \cite{Wil}.

The {\em degree} $\deg_G(p)$ of a vertex $p\in \pi(G)$ in ${\rm
GK}(G)$ is the number of edges incident on $p$. If $\pi(G)=\{p_1,
p_2, \ldots, p_h\}$ with $p_1<p_2<\cdots<p_h$, then we define
\[ {\rm D}(G):=\big(\deg_G(p_1), \deg_G(p_2), \ldots, \deg_G(p_h)\big), \]
which is called the {\em degree pattern of $G$}.

Given a group $G$, denote by $h_{\rm OD}(G)$ the number of
isomorphism classes of groups with the same order and degree
pattern as $G$. For instance, if $G$ is a cyclic group of order
$p$, where $p$ is a prime number, then $h_{\rm OD}(G)=1$. Notice
that the cyclic group $\mathbb{Z}_p$ is the only group of order
$p$. Similarly, in the case when $G$ is a $p$-group of order
$p^2$, $h_{\rm OD}(G)=2$, in fact, the non-isomorphic groups
$\mathbb{Z}_{p^2}$ and $\mathbb{Z}_p\times \mathbb{Z}_p$ are the
only groups with the same order and degree pattern. All finite
groups, in terms of the function $h_{\rm OD}(\cdot)$, are
classified as follows:
\begin{definition} A group $G$ is called  $k$-fold
OD-characterizable if $h_{\rm OD}(G)=k$. Usually, a $1$-fold
OD-characterizable group is simply called OD-characterizable.
\end{definition}

There are scattered results in the literature showing that
certain simple groups are $k$-fold OD-characterizable for $k\in
\{1, 2\}$ (see Table 1). In this article, we will focus our
attention on the symplectic groups $C_p(2)\cong S_{2p}(2)$, where
$p$ is an odd prime. Recall that $C_2(2)$ is not a simple group,
in fact, the derived subgroup $C_2(2)'$ is a simple group which is
isomorphic with $A_6\cong L_2(9)$. In addition, we recall that
$B_2(3)\cong {^2\!A}_4(2^2)$, $B_n(2^m)\cong C_n(2^m)$ and
$B_2(q)\cong C_2(q)$ (see \cite{carter}). Previously, it was
determined the values of $h_{OD}(\cdot)$ for some sympelectic and
orthogonal groups (see \cite{akbarim, moghadam}):

\begin{center}
$\begin{array}{llllc} \hline \ \ \ \ \ \ \ \ G &&& \mbox{Restrictions on} \ G & h_{OD}(G) \\
\hline
B_3(4)\cong C_3(4) & &&& 1\\[0.2cm]
B_2(q)\cong C_2(q)&& & |\pi(\frac{q^2+1}{(2, q-1)})|=1 &
1\\[0.2cm]
B_{2^m}(q)\cong C_{2^m}(q)&& & |\pi(\frac{q^2+1}{(2,
q-1)})|=1, \ \ q \ \mbox{is even}  & 1\\[0.2cm]
B_3(5), \ C_3(5),&& & & 2 \\[0.2cm]
B_n(q), \ C_n(q),&& &  n=2^m>2, \ \  |\pi(\frac{q^n+1}{2})|=1, \ \ q \ \mbox{is an odd prime power}& 2\\[0.3cm]
B_p(3), \ C_p(3),&& & |\pi(\frac{3^p-1}{2})|=1, \ \ \ \ p \
\mbox{is an odd prime,} \ \ \ & 2 \\
\hline
\end{array}$
\end{center}

If $n$ is a natural number, then $\pi(n)$ denotes the set of all
prime divisors of $n$. Given a group $G$, the order of $G$ can be
expressed as a product of some co-prime natural numbers $m_i(G)$,
$i=1, 2, \ldots, t(G)$, with $\pi(m_i(G))=\pi_i(G)$. The numbers
$m_1(G), m_2(G), \ldots, m_{t(G)}(G)$ are called the {\it order
components of} $G$. We set $${\rm OC}(G):=\{m_1(G), m_2(G),
\ldots, m_{t(G)}(G)\}.$$ A list of order components of simple
groups with disconnected prime graphs can be found in
\cite[Tables 1-4 ]{c1}. In the similar manner, we define $h_{\rm
OC}(G)$ as the number of isomorphism classes of finite groups with
the same set ${\rm OC}(G)$ of order components. Again, in terms of
function $h_{\rm OC}(\cdot)$, the groups $G$ are classified as
follows:
\begin{definition} A finite group $G$ is called $k$-fold
OC-characterizable if $h_{\rm OC}(G)=k$. In the case when $k=1$
the group $G$ is simply called OC-characterizable.
\end{definition}

It is worth mentioning that the characterization of finite groups
through their order components was first introduced by Chen in
\cite{c1}. A Mersenne prime is a prime that can be written as
$2^p-1$ for some prime $p$. The purpose of this article is to
prove the following theorem.

\begin{theorem}\label{th1}\label{th1}
Let $C=C_p(2)$ be the symplectic group over binary field, for
which $2^p-1>7$ is a Mersenne prime. Then $h_{\rm OD}(C)=1=h_{\rm
OC}(C)$.
\end{theorem}

{\em Remark}. It is worth noting that the values of functions
$h_{\rm OD}(\cdot)$ and $h_{\rm OC}(\cdot)$ may be different. For
instance, suppose $M\in \{B_3(5), C_3(5)\}$. By \cite{Wil}, the
prime graph associated with $M$ is connected and so ${\rm
OC}(M)=\{|M|=2^9\cdot 3^4\cdot 5^9\cdot 7\cdot 13\cdot 31\}$. On
the other hand, it is easy to see that the prime graph associated
with a nilpotent group is always a clique, hence, we have
$$h_{\rm OC}(M)>\nu_{\rm nil}(|M|)\geqslant \nu_{\rm a}(|M|)={\rm Par}(9)^2\cdot {\rm Par}(4)=30^2\times 5=4500,$$
where $\nu_{\rm nil}(n)$ (resp. $\nu_{\rm a}(n)$) signifies the
number of non-isomorphic nilpotent (resp. abelian) groups of
order $n$ and ${\rm Par}(n)$ denotes the number of partitions of
$n$. However, by Theorem 1.3 in \cite{akbarim}, we know that
$h_{\rm OD}(M)=2$.

We now introduce some further notation. If $a$ is a natural
number, $r$ is an odd prime and $(r, a)=1$, then by $e(r, a)$ we
denote the multiplicative order of $a$ modulo $r$, that is the
minimal natural number $n$ with $a^n\equiv 1\pmod{r}$. If $a$ is
odd, we put $$e(2, a)=\left \{\begin{array}{ll} 1 & \ \ {\rm if}
\ a\equiv 1\pmod{4},
\\[0.3cm] 2 & \ \ {\rm if} \ a\equiv -1\pmod{4}.
\end{array}\right.
$$ We also define the function $\eta: \mathbb{N}\longrightarrow
\mathbb{N}$, as follows

$$\eta(m)=\left \{\begin{array}{ll} m & \ \ {\rm if} \ m\equiv
1\pmod{2},
\\[0.3cm] \frac{m}{2} & \ \ {\rm if} \ m\equiv 0\pmod{2}.
\end{array}\right.
$$
Moreover, we will use the notation $\mathbb{A}_n$ and
$\mathbb{S}_n$ to denote the alternating and the symmetric group
on $n$ letters, respectively. All unexplained notation and
terminology are borrowed from the Atlas \cite{atlas}.
\section{Preliminaries}
The following lemma is a consequence of Zsigmondy's theorem (see
\cite{zsigmondy}).
\begin{lm}\label{zsig}
Let $a$ be a natural number greater than $1$. Then for every natural number $n$
there exists a prime $r$ with $e(r, a)=n$ but for
the cases $(n, a)\in \{(1, 2), (1, 3), (6,2)\}$
\end{lm}

A prime $r$ with $e(r,a)=n$ is called a {\em primitive prime
divisor} of $a^n-1$. By Lemma \ref{zsig}, such a prime exists
except for the cases mentioned in the lemma. Given a natural
number $a$, we denote by ${\rm ppd}(a^n-1)$ the set of all
primitive prime divisors of $a^n-1$. By our definition, we have
$\pi(a-1)={\rm ppd}(a-1)$ but for the following sole exception,
namely, $2\notin {\rm ppd}(a-1)$ if $e(2, a)=2$. In this case, we
assume that $2\in {\rm ppd}(a^2-1)$.

\begin{lm}\label{l2} (\cite{vvc})
Let $M$ be one of the simple groups of Lie type, $B_n(q)$ or
$C_n(q)$, over a field of characteristic $p$. Let $r$, $s$ be odd
primes with $r, s\in \pi(M)\setminus \{p\}$. suppose that $r\in
{\rm ppd}(q^k-1)$, $s\in {\rm ppd}(q^l-1)$ and $1\leq \eta(k)\leq
\eta(l)$. Then $r$ and $s$ are non-adjacent if and only if
$\eta(k)+\eta(l)>n$ and $\frac{l}{k}$ is not an odd natural
number.
\end{lm}

\begin{lm}\label{l3}(\cite{vv}) Let $M$ be one of the simple groups of Lie type, $B_n(q)$ or $C_n(q)$, over a field
of characteristic $p$, and let $r\in \pi(M)\setminus\{p\}$ and
$r\in {\rm ppd}(q^k-1)$. Then $r$ and $p$ are non-adjacent if and
only if $\eta(k)>n-1$.
\end{lm}

Using Lemmas \ref{l2} and \ref{l3}, we have:
\begin{itemize}
\item  The prime graphs ${\rm GK}(B_n(q))$ and ${\rm GK}(C_n(q))$ coincide {\rm \cite[Proposition 7.5]{vv}}.
\item $|B_n(q)|=|C_n(q)|$ and ${\rm D}(B_n(q))={\rm D}(C_n(q))$.
\end{itemize}
\begin{corollary}\label{2od} Let $p>3$ be a prime and $C=C_p(2)$.
Then $\deg_C(3)=|\pi_1(C)|-1$.
\end{corollary}
\begin{proof}
Recall that
$$\pi_1(M)=\pi\Big(2(2^p+1)\prod_{i=1}^{p-1}(2^{2i}-1)\Big) \ \ \
\mbox{and} \ \ \ \pi_2(M)=\pi(2^p-1).$$ Moreover, by Lemma
\ref{l2}, it follows that only primitive prime divisors of
$2^{p}-1$ are non-adjacent to $3$. But by Lemmas \ref{l2} and
\ref{l3} we deduce that $\deg(3)=|\pi_1(M)|-1$, as desired.
\end{proof}

The following easy lemma (which is appeared in \cite{Gui}) is
crucial to the study of characterizability of symplectic groups
$C_p(2)$ by order components.
\begin{lm}\label{l4} Let $G$ be a group whose prime graph has
more than one component. If $H$ is a normal $\pi_k$-subgroup of
$G$, then $|H|-1$ is divisible by $m_i$, $i\neq k$.
\end{lm}

\begin{lm}\label{l5} (\cite{Moghaddam})
Let $S$ be a simple group with a disconnected prime graph ${\rm
GK}(S)$, except $U_4(2)$ and $U_5(2)$. If $G$ is a group with
${\rm OC}(G)={\rm OC}(S)$, then $G$ is neither a Frobenius group
nor a $2$-Frobenius group.
\end{lm}
\begin{lm}\label{l6} (\cite{Wil})
Let $G$ be a group with $t(G)\geq 2$. Then one of the following
hold:
\begin{itemize}
\item[$(1)$] $G$ is either a Frobenius group or a $2$-Frobenius group.
\item[$(2)$] $G$ has a normal series $1\unlhd H\lhd K\unlhd G$ such that $H$ is a nilpotent $\pi_1$-group,
$K/H$ is a non-abelian simple group, $G/K$ is a $\pi_1$-group,
$|G/K|$ divides $|{\rm Out}(K/H)|$ and any odd order component of
G is equal to one of the odd order components of $K/H$.
\end{itemize}
\end{lm}

\begin{lm}\label{l7} (\cite{Cre})
The only solution of the equation $p^m-q^n = 1$, where $p, q$ are
primes and $m, n>1$ are integers, is $(p, q, m, n)=(3, 2, 2, 3)$.
\end{lm}

Given a natural number $B$ and a prime number $t$, we denote by
$B_t$ the $t$-part of $B$, that is the largest power of $t$
dividing $B$.

\begin{lm}\label{l8} (\cite{Chen})
Let $B=(2^2-1)(2^4-1)\cdots (2^{2n}-1)$. If $t$ is a prime divisor
of $B$, then $B_t<2^{3n}$. Furthermore, if $t\geq 5$ then
$B_t<2^{2n}$.
\end{lm}


\section{Proof of Theorem \ref{th1}}
Throughout this section, we will assume that $2^p-1>7$ is a
Mersenne prime and $C=C_p(2)$. Suppose that $G$ is a group with
the same order and degree pattern as $C$, that is
$$|G|=|C|=2^{p^2}\prod_{i=1}^{p}(2^{2i}-1) \ \ \mbox{and} \ \ {\rm D}(G)={\rm D}(C).$$
Note that, according to the results summarized in \cite{k}, we
have
$$\pi_1(C)=\pi\Big(2(2^p+1)\prod_{i=1}^{p-1}(2^{2i}-1)\Big) \ \ \
\mbox{and} \ \ \ \pi_2(C)=\{2^p-1\}.$$ By our hypothesis, it is
easy to see that
$$ \pi_2(G)=\pi_2(C)=\{2^p-1\} \ \ {\rm
and} \ \ \pi(G)=\pi(C)=\pi_1(C)\cup \{2^p-1\},$$ and so $t(C)=2$.
First of all, we notice that $2^p-1$ is the largest prime in
$\pi(G)=\pi(C)$. Moreover, it follows from Corollary \ref{2od}
that
$${\rm deg}_G(3)=\deg_C(3)=|\pi_1(C)|-1,$$ and this forces
$\pi_1(G)=\pi_1(C)$ and $t(G)=2$. Hence, we have $${\rm
OC}(G)={\rm
OC}(C)=\left\{2^{p^2}(2^p+1)\prod_{i=1}^{p-1}(2^{2i}-1), \ \
2^p-1\right\},$$ and from Lemma \ref{l5}, the group $G$ is neither
a Frobenius group nor a $2$-Frobenius group. Finally, Lemma
\ref{l6}, reduces the problem to the study of the simple groups.
Indeed, by Lemma \ref{l6}, there is a normal series $1\unlhd
H\lhd K\unlhd G$ of $G$ such that:
\begin{itemize}
\item[$(1)$]  $H$ is a nilpotent $\pi_1$-group, $K/H$ is a non-abelian
simple group and $G/K$ is a  $\pi_1$-group. Moreover, we have
$$K/H\leqslant G/H \leqslant {\rm Aut}(K/H),$$ and $t(K/H)\geqslant
t(G)\geqslant 2$,
\item[$(2)$] $2^p-1$ is the only odd order component of $G$ which
is equal to one of those of the qutient $K/H$,
\item[$(3)$] $|G/K|$ divides $|{\rm Out}(K/H)|$.
\end{itemize}

Now we will continue the proof step by step.

\begin{step} $K/H\ncong {^2A_3(2)}, {^2F_4(2)'}, {^2A_5(2)}, E_7(2), E_7(3), A_2(4), {^2E_6(2)}$
 or one of the sporadic simple groups. \end{step}

Note that either the odd order components of above groups are not
equal to a Mersenne prime $2^p-1>7$ or their orders does not
divide the order of $G$.

\begin{step} $K/H\ncong \mathbb{A}_n$, where $n$ and $n-2$ are both prime numbers.
\end{step}
In this case, it follows that $n=2^p-1$. Now, simple computations
show that
$$\left(\frac{n!}{2}\right)_2=2^{\big(\big[\frac{n}{2}\big]+
\big[\frac{n}{2^2}\big]+ \cdots \big)-1}=2^{2^p-p-2}.$$

Now, if $p>5$, then $2^p-p-2>p^2$ and hence the $2$-part of
$|A_n|$ does not divide the $2$-part of $|G|$, a contradiction.
In the case when $p=5$, then $n=31$ and $|K/H|=(31!)/2$, which
does not divide $|G|=|C_5(2)|=2^{25}\cdot 3^6\cdot 5^2\cdot
7\cdot 11\cdot 17\cdot 31$, again a contradiction.

\begin{step} $K/H\ncong \mathbb{A}_n$, where $n=q, q+1, q+2$
($q$ is a prime number), and one of $n$, $n-2$ is not prime.
\end{step}

Here, $\{q\}$ is the only odd order component of $K/H$, and so
$q=2^p-1$. Now, we consider the alternating group $\mathbb{A}_q$
which is a subgroup of $K/H\cong \mathbb{A}_n$. Similar arguments
as those in the previous step, on the subgroup $\mathbb{A}_q$
instead of $\mathbb{A}_n$, lead us a contradiction.
\begin{step}
$K/H$ is isomorphic neither ${^2E}_6(q)$, $q>2$, or $E_6(q)$.
\end{step}
We deal with ${^2E}_6(q)$, $q>2$, the proof for $E_6(q)$ being
quite similar. Suppose $K/H\cong {^2E}_6(q)$. First of all, we
recall that
$$|{^2E}_6(q)|=\frac{1}{(3,q+1)}q^{36}(q^{12}-1)(q^9+1)(q^8-1)(q^6-1)(q^5+1)(q^2-1).$$
Considering the only odd order component of ${^2E}_6(q)$, that is
$(q^6-q^3+1)/(3, q+1)$, we must have $(q^6-q^3+1)/(3,
q+1)=2^p-1$, which implies that $q^9>2^p$, or equivalently
$q^{36}>2^{4p}$. Let $q=r^f$. If $r$ is an odd prime, then from
Lemma \ref{l8}, we get
$$q^{36}=r^{36f}=|K/H|_r\leqslant |G|_r<2^{3p},$$
which is a contradiction. Therefore we may assume that $r=2$. In
this case, we have $(2^{6f}-2^{3f}+1)/(3, 2^f+1)=2^p-1$. Now, if
$(3, 2^f+1)=1$, then we obtain $2^{3f}(2^{3f}-1)=2(2^{p-1}-1)$,
from which we deduce that $3f=1$, a contradiction. In the case
where $(3, 2^f+1)=3$, an easy calculation shows that
$2^{3f}(2^{3f}-1)=2^2(3\cdot 2^{p-2}-1)$, and so $3f=2$, which is
again a contradiction.
\begin{step}
$K/H\ncong F_4(q)$, where $q$ is an odd prime power.
\end{step}
We remark that $q^4-q^2+1$ is the only odd order component of
$F_4(q)$, and clearly this forces $q^4-q^2+1=2^p-1$. Then
$q^2(q^2-1)=2(2^{p-1}-1)$, which shows that $2(2^{p-1}-1)$ is
divisible by $4$, a contradiction.

\begin{step}
$K/H\ncong {^2B}_2(q)$, where $q=2^{2m+1}>2$.
\end{step}
Recall that $|{^2B}_2(q)|=q^{2}(q^2+1)(q-1)$ and the odd order
components of ${^2B}_2(q)$ are
$$q-1, \ \ \ q-\sqrt{2q}+1, \ \ \ q+\sqrt{2q}+1.$$

If $q-1=2^p-1$, then $q=2^p$. Now, we consider the primitive
prime divisor $r\in {\rm ppd}(2^{4p}-1)$. Clearly $r\in
\pi(2^{2p}+1)$, and so $r\in \pi({^2B}_2(q))\subseteq \pi(G)$.
This is a contradiction.

In the case when $q-\sqrt{2q}+1=2^p-1$ (resp.
$q+\sqrt{2q}+1=2^p-1$), by simple computations we obtain
$2^{m+1}(2^m-1)=2(2^{p-1}-1)$ (resp.
$2^{m+1}(2^m+1)=2(2^{p-1}-1)$), a contradiction.
\begin{step}
$K/H\ncong E_8(q)$, where $q\equiv 2, 3 \pmod{5}$.
\end{step}
The odd order components of $E_8(q)$ in this case are
$$q^8-q^4+1, \ \ \ \ \ \frac{q^{10}+q^5+1}{q^2+q+1}, \ \ \ \ \ \frac{q^{10}-q^5+1}{q^2-q+1}.$$

If $q^8-q^4+1=2^p-1$, then we obtain
$q^4(q-1)(q+1)(q^2+1)=2(2^{p-1}-1)$. However, the left hand side
is divisible by 16, while the right hand side is divisible by 2,
an impossible.

If $(q^{10}+ q^5+1)/(q^2+ q+1)=2^p-1$, then after subtracting 1
from both sides of this equation and some simple computations, we
obtain
$$q(q-1)(q+1)(q^2+1)(q^3+q^2-1)=2(2^{p-1}-1).$$ Now, if $q$ is odd, then
the left hand side is divisible by 16, a contradiction. Moreover,
if $q$ is even, then it follows that $q=2$, and if this is
substituted in above equation we get $83=2^{p-2}$, a
contradiction.

The case $(q^{10}-q^5+1)/(q^2-q+1)=2^p-1$ is quite similar to the
previous case and it is omitted.
\begin{step}
$K/H\ncong E_8(q)$, where $q\equiv 0, 1, 4 \pmod{5}$.
\end{step}
The odd order components of $E_8(q)$ in this case are
$$\frac{q^{10}+1}{q^2+1}, \ \ \ q^8-q^4+1, \ \ \
\frac{q^{10}+q^5+1}{q^2+q+1}, \ \ \ \frac{q^{10}-q^5+1}{q^2-
q+1}.$$

Consider the first case. Let $(q^{10}+1)/(q^2+1)=2^p-1$.
Subtracting 1 from both sides of this equality, we get
$$q^2(q^2-1)(q^4+1)=2(2^{p-1}-1),$$
which implies $2(2^{p-1}-1)$ is divisible by $4$, a contradiction.

Similarly, if $q^8-q^4+1=2^p-1$, we obtain
$q^4(q-1)(q+1)(q^2+1)=2(2^{p-1}-1)$, which shows that
$2(2^{p-1}-1)$ is divisible by $16$, a contradiction.

Similar arguments work if $(q^{10}+ q^5+1)/(q^2+ q+1)=2^p-1$ or
$(q^{10}- q^5+1)/(q^2- q+1)=2^p-1$, and we omit the details.
\begin{step}
$K/H\ncong {^2F}_4(q)$, where $q=2^{2m+1}>2$.
\end{step}
The odd order components of ${^2F_4(q)}$ are
$q^2+\sqrt{2q^3}+q+\sqrt{2q}+1$ and
$q^2-\sqrt{2q^3}+q-\sqrt{2q}+1$. Therefore, we have
$q^2+\sqrt{2q^3}+q+\sqrt{2q}+1=2^p-1$ or
$q^2-\sqrt{2q^3}+q-\sqrt{2q}+1=2^p-1$. However, if $2^{2m+1}$ is
substituted in these equations we obtain $2^{m+1}(2^{3m+1}\pm
2^{2m+1}\pm 2^m \pm 1)=2(2^{p-1}-1)$, which is a contradiction.

\begin{step}
$K/H\ncong F_4(q)$, where $q=2^m$.
\end{step}
The odd order components of $F_4(q)$ are $q^4+1$ and $q^4-q^2+1$.
It is easy to see that in both cases, $2^{2m}$ divides
$2(2^{p-1}-1)$, a contradiction.

\begin{step}
$K/H\ncong {^2G_2(q)}$, where $q=3^{2m+1}>3$.
\end{step}
The odd order components of ${^2G_2(q)}$ are $q+\sqrt{3q}+1$ and
$q-\sqrt{3q}+1$. If $q-\sqrt{3q}+1=2^p-1$, then $q^3>2^{3p}$,
while Lemma \ref{l8} shows that $q^3<2^{3p}$, which is a
contradiction. If $q+\sqrt{3q}+1=2^p-1$, then
\begin{equation}\label{e1}
2^p-2= 2(2^{\frac{p-1}{2}}-1)(2^{\frac{p-1}{2}}+1)=3^{m+1}(3^m+1).
\end{equation}
First of all, we recall that $(2^{\frac{p-1}{2}}-1,
2^{\frac{p-1}{2}}+1)=1$. Now we consider two cases separately:
\begin{itemize}
\item[$(i)$] If $3^{m+1}$ divides $2^{\frac{p-1}{2}}-1$, then
 $$3^m+1 < 3^{m+1}\leq 2^{\frac{p-1}{2}}-1 <
2^{\frac{p-1}{2}}+1.$$
Hence, we obtain
$$3^{m+1}(3^m+1)<2(2^{\frac{p-1}{2}}-1)(2^{\frac{p-1}{2}}+1),$$
a contradiction.
\item[$(ii)$] If $3^{m+1}$ divides $2^{\frac{p-1}{2}}+1$, then $2^{\frac{p-1}{2}}+1=k \cdot 3^{m+1}$
where $k$ is a natural number. Now, from Eq. (\ref{e1}), it
follows that $$2k(2^{\frac{p-1}{2}}-1)=3^m+1,$$ and consequently
$3^m\geq 2^{\frac{p+1}{2}}-1$. Therefore we have
$$2^{\frac{p+1}{2}}-1\leq 3^m<3^{m+1}\leq 2^{\frac{p-1}{2}}+1,$$
a contradiction.
\end{itemize}
\begin{step}
$K/H\ncong G_2(q)$, where $q=3^m$.
\end{step}
Recall that the odd order components of $G_2(q)$ are $q^2-q+1$ and
$q^2+q+1$. If $q^2-q+1=2^p-1$ then $q^6>2^{3p}$, while one can
follow from Lemma \ref{l8} that $q^6<2^{3p}$, which is a
contradiction. If $q^2+q+1=2^p-1$, then $q(q+1)\equiv 2
\pmod{4}$, which forces $m$ is even. But then, it is obvious that
$2^p-2=q(q+1)\equiv 2\pmod{8}$, a contradiction.
\begin{step}
$K/H\ncong {^2D}_r(3)$, where $r=2^m+1$ is a prime number and
$m\geq 1$.
\end{step}
Recall that
$$|{^2D_r(3)}|=\frac{1}{(4, 3^r+1)}3^{r(r-1)}(3^r+1)\prod\limits_{i=1}^{r-1}(3^{2i}-1),$$
and the odd order components of ${^2D_r(3)}$ are
$$\frac{3^{r-1}+1}{2} \ \ \ \  \mbox{and} \ \ \ \ \  \frac{3^r+1}{4}.$$
In the case when $(3^{r-1}+1)/2=2^p-1$, adding 1 to both sides of
this equality, we obtain $3(3^{r-2}+1)=2^{p+1}$, which is a
contradiction. If $(3^r+1)/4=2^p-1$, then $r\geqslant 5$ because
$p\geqslant 5$. Moreover, on the one hand, from last equation we
obtain $3^r=2^{p+2}-5>2^{p+1}$, which implies that
$$3^{r(r-1)}>2^{(p+1)(r-1)}>2^{4(p+1)}.$$
On the other hand, it follows from Lemma \ref{l8} that
$$3^{r(r-1)}=|K/H|_3\leqslant |G|_3<2^{3p},$$
which is a contradiction.
\begin{step}
$K/H\ncong B_n(q)$, where $n=2^m\geq 4$ and $q=r^f$ is an odd
prime power.
\end{step}
Note that
$$|B_n(q)|=\frac{1}{(2,q-1)}q^{n^2}\prod\limits^n_{i=1}(q^{2i}-1),$$
and the only odd order component of $B_n(q)$ is $(q^n+1)/2$. If
$(q^n+1)/2=2^p-1$, then $q^n=2^{p+1}-3>2^p$ and clearly $q$ is not
divisible by $2$ and $3$. Since $p\geq 5$ and $n\geqslant 4$, it
is easy to see that
$$q^{n^2}>q^{3n}>2^{3p}>2^{2p}.$$  On the other hand, by Lemma \ref{l8}, we
obtain $$q^{n^2}=|K/H|_r\leqslant |G|_r<2^{2p},$$ which is a
contradiction.
\begin{step}
$K/H\ncong B_r(3)$.
\end{step}
The only odd order component of $B_r(3)$ is $(3^r-1)/2$. If
$(3^r-1)/2=2^p-1$, then $2^{p+1}-3^r=1$. However, this equation
has no solution by Lemma \ref{l7}, which is impossible.
\begin{step}
$K/H\ncong {^3D}_4(q)$.
\end{step}
We recall that $q^4-q^2+1$ is the only odd order component of
${^3D}_4(q)$, and so $q^4-q^2+1=2^p-1$. But then,
$q^2(q^2-1)=2(2^{p-1}-1)$, which shows that $2(2^{p-1}-1)$ is
divisible by $4$, a contradiction.
\begin{step}
$K/H\ncong G_2(q)$, where $2<q\equiv \pm 1 \pmod{3}$.
\end{step}
In this case, the odd order components of $G_2(q)$ are $q^2+q+1$
and $q^2-q+1$. Let $q=r^f$. If $q^2+q+1=2^p-1$, then
$q(q+1)=2(2^{p-1}-1)$, which shows that $q>2$ is not a power of
$2$. Moreover, since $q-1\geqslant 2$, we obtain
$$q^3-1=(q-1)(q^2+q+1)\geqslant 2(2^p-1),$$
and so $q^3\geqslant 2^{p+1}-1>2^p$, which yields that $q^6>
2^{2p}$. However, since $|G_2(q)|=q^6(q^2-1)(q^6-1)$, from Lemma
\ref{l8}, we conclude that $$q^6=|K/H|_r\leqslant
|G|_r<2^{2p},$$  which is a contradiction.

The case when $q^2-q+1=2^p-1$ is similar and left to the reader.
\begin{step}
$K/H\ncong {^2D}_n(3)$, where $n=2^m+1$ which is not a prime and
$m\geq 2$.
\end{step}
The odd order component of ${^2D}_n(3)$ is $(3^{n-1}+1)/2$. If
$(3^{n-1}+1)/2=2^p-1$, then $2^{p+1}=3(3^{n-2}-1)$, a
contradiction.
\begin{step}
$K/H\ncong {^2D}_r(3)$, where $r\geqslant 5$ is a prime and
$r\neq 2^m+1$.
\end{step}
We recall that
$$|{^2D}_r(3)|=\frac{1}{(4,
3^r+1)}3^{r(r-1)}(3^r+1)\prod\limits_{i=1}^{r-1}(3^{2i}-1).$$
Moreover, the only odd order component of ${^2D}_r(3)$ is
$(3^r+1)/4$. Let $(3^r+1)/4=2^p-1$. An easy computation shows that
$3^r=2^{p+2}-5>2^{p+1}$. Moreover, we note that $r-1\geqslant 4$,
and so
 $$3^{r(r-1)}\geqslant 3^{4r}>2^{4(p+1)}.$$ On the other hand, by Lemma
\ref{l8}, we obtain  $$3^{r(r-1)}=|K/H|_3\leqslant |G|_3<2^{3p},$$
which is a contradiction.
\begin{step}
$K/H\ncong {^2D}_n(2)$, where $n=2^m+1$, $m\geq 2$.
\end{step}
The only odd order component of ${^2D}_n(2)$ is $2^{n-1}-1$. If
$2^{n-1}-1=2^p-1$, then $n-1=p$ and $2^m=p$, an impossible.
\begin{step}
$K/H\ncong {^2D}_n(q)$, where $n=2^m\geq 4$ and $q=r^f$.
\end{step}
Recall that
$$|^2D_n(q)|=\frac{1}{(4,q^n+1)}q^{n(n-1)}(q^n+1)\prod\limits_{i=1}^{n-1}(q^{2i}-1),$$
and the only odd order component of ${^2D}_n(q)$ is $(q^n+1)/(2,
q+1)$. Then $(q^n+1)/(2, q+1)=2^p-1$. Assume first that $(2,
q+1)=1$. In this case, we obtain $q^n=2(2^{p-1}-1)$, a
contradiction. Assume next that $(2, q+1)=2$. Again, using simple
calculations we obtain $q^n=2^{p+1}-3>2^p$ and so $q$ cannot be a
power of $2$. Thus, since $n-1\geqslant 3$, $q^{n(n-1)}\geqslant
q^{3n}>2^{3p}$. However, Lemma \ref{l8} shows that
$$q^{n(n-1)}=|K/H|_r\leqslant |G|_r<2^{3p},$$ which is a
contradiction.
\begin{step}
$K/H\ncong D_{r+1}(q)$, where $q=2, 3$.
\end{step}
The only odd order component of $D_{r+1}(q)$ is $(q^{r}-1)/(2,
q-1)$, and so $(q^r-1)/(2, q-1)=2^p-1$. If $(2, q-1)=1$, then
$r=p$ and $q=2$, and we have $$|K/H|=|D_{p+1}(2)|=\frac{1}{(4,
2^{p+1}-1)}2^{p(p+1)}(2^{p+1}-1)\prod_{i=1}^{p}(2^{2i}-1),$$ this
shows that $|K/H|_2$ does not divide $|G|_2$, which is a
contradiction. In the case when $(2, q-1)=2$, we have the equation
$2^{p+1}-3^r=1$, which has no solution for $p\geqslant 5$, by
Lemma \ref{l7}. This is the final contradiction.
\begin{step}
$K/H\ncong D_r(q)$, where $q=2, 3, 5$ and $r\geq 5$.
\end{step}
We recall that the only odd order component of $D_r(q)$ is
$(q^r-1)/(q-1)$. We distinguish three cases separately.
\begin{itemize}
\item[$(i)$] $q=2$. In this case, we have $2^r-1=2^p-1$, and so $r=p$ and
$$|K/H|=|D_{p}(2)|=2^{p(p-1)}(2^p-1)\prod_{i=1}^{p-1}(2^{2i}-1).$$ Note that $|{\rm Out}(D_p(2))|=2$ and
$D_p(2)\leqslant G/H\leqslant {\rm Aut}(D_p(2))$. Now,
considering the order of groups, we get $|H|=2^\alpha(2^p+1)$
where $p-1\leq \alpha\leq p$. Let $r\in {\rm ppd}(2^{2p}-1)$ and
$Q\in {\rm Syl}_r(H)$. Clearly $r\in \pi(2^p+1)$, $Q$ is a normal
$\pi_1$-subgroup of $G$ and $|Q|$ divides $2^p+1$. Now, from Lemma
\ref{l4}, it follows that $|Q|-1$ is divisible by $m_2(G)=2^p-1$,
and so $|Q|-1\geqslant 2^p-1$ or equivalently $|Q|\geqslant 2^p$.
This forces $|Q|=2^p+1$. But then $m_2(G)=2^p-1$ does not divide
the value $|Q|-1=2^p$, which is a contradiction.
\item[$(ii)$] $q=3$. In this case, from the equality $(3^r-1)/2=2^p-1$, we deduce that $2^{p+1}-3^r=1$.
However, this equation has no solution when $p\geqslant 5$ by
Lemma \ref{l7}, a contradiction.
\item[$(iii)$] $q=5$. Here $(5^r-1)/4=2^p-1$, and so $5^r=2^{p+2}-3>2^{p+1}$. As before, since $r-1\geqslant 4$,
we obtain $5^{r(r-1)}> 5^{4r}> 2^{4(p+1)}$. On the other hand, we
have
$$5^{r(r-1)}=|K/H|_5\leqslant |G|_5< 2^{2p},$$
by Lemma \ref{l8}, which is a contradiction.
\end{itemize}
\begin{step}
$K/H\ncong C_r(3)$.
\end{step}
The only odd order component of $C_r(3)$ is $(3^r-1)/2$. Thus, if
$(3^r-1)/2=2^{p}-1$, then $2^{p+1}-3^r=1$. However, this equation
has no solution by Lemma \ref{l7}, an impossible.
\begin{step}
$K/H\ncong C_n(q)$, where $n=2^m\geq 2$.
\end{step}
Note that
$$|C_n(q)|=\frac{1}{(2,q-1)}q^{n^2}\prod\limits^n_{i=1}(q^{2i}-1),$$
and the only odd order component of $C_n(q)$ is $(q^n+1)/(2,
q-1)$. Thus $(q^n+1)/(2, q-1)=2^p-1$. If $(2, q-1)=1$, then
$q^n=2(2^{p-1}-1)$, which yields that $q=p=2$ and $n=1$, a
contradiction. If $(2, q-1)=2$, then $q^n=2^{p+1}-3>2^p$, which
implies that $q$ is not a power of $2$ and $3$. Let $q=r^f$. When
$n\geq 4$, it is easy to see that
$$q^{n^2}>q^{3n}>2^{3p}>2^{2(p+1)}.$$ But, from Lemma \ref{l8}, we obtain
$$q^{n^2}=|K/H|_r\leqslant |G|_r<2^{2p},$$
a contradiction. Assume now that $n=2$. In this case, we have
$q^2=2^{p+1}-3$, or equivalently
$$(q-1)(q+1)=2^2(2^{p-1}-1).$$
However, the left hand side is divisible by $8$, while the right
hand side is divisible by 4, a contradiction.
\begin{step}
$K/H\ncong A_1(q)$, where $q=2^m>2$.
\end{step}
The odd order components of $A_{1}(q)$ are $q+1$ and $q-1$. If
$q+1=2^p-1$, then $q=2(2^{p-1}-1)$, a contradiction. If
$q-1=2^p-1$, then $q=2^p$. Moreover, since $A_1(q)\leqslant
G/H\leqslant {\rm Aut}(A_1(q))$, it is easy to see that the order
of $H$ is divisible by $(2^2-1)(2^4-1)\cdots(2^{2(p-1)}-1)$. Let
$r\in {\rm ppd}(2^{2(p-1)}-1)$ and $Q\in {\rm Syl}_r(H)$. Clearly
$Q$ is a $\pi_1$-normal subgroup of $G$ and $|Q|$ divides
$2^{p-1}+1$. On the other hand, from Lemma \ref{l4}, $|Q|-1$ is
divisible by $2^p-1$ which implies that $|Q|\geqslant 2^p$. This
is a contradiction.
\begin{step}
$K/H\ncong A_{1}(q)$, where $3\leqslant q\equiv \pm 1\pmod{4}$ and
$q=r^f$.
\end{step}
Assume first that $3\leqslant q\equiv 1\pmod{4}$.  In this case,
the odd order components of $A_{1}(q)$ are $(q+1)/2$ and $q$. If
$(q+1)/2=2^p-1$, then $r^f=q=2^{p+1}-3$. First of all, we claim
that $f$ is an odd number. Otherwise, we have
$$(r^{\frac{f}{2}}-1)(r^{\frac{f}{2}}+1)=2^2(2^{p-1}-1).$$
But then, the left hand side is divisible by 8, while the right
hand side is divisible by 4, which is a contradiction.
Furthermore, by easy computations we observe that
$$|A_1(q)|=\frac{1}{2}q(q^2-1)=2^2(2^{p+1}-3)(2^{p-1}-1)(2^p-1).$$ On the
other hand, since $|G/K| \cdot |H| =|G|/|A_1(q)|$, we deduce that
$$|G/K|_2 \cdot |H|_2=\frac{|G|_2}{|A_1(q)|_2}=2^{p^2-2}.$$
But since $|G/K|$ divides $|{\rm Out}(A_1(q))|=2f$ and $f$ is odd,
$|G/K|_2$ is at most $2$. Hence, if $S_2\in {\rm Syl}_2(H)$, then
$|S_2|=2^{p^2-2}$ or $|S_2|=2^{p^2-3}$. We notice that $S_2$ is a
normal subgroup of $G$, because $H$ is nilpotent. Now, it follows
from Lemma \ref{l4} that $2^p-1$ divides $2^{p^2-2}-1$ or
$2^{p^2-3}-1$, which is a contradiction. If $q=2^p-1$, we get a
contradiction by Lemma \ref{l7}.

Assume next that $3\leqslant q\equiv -1\pmod{4}$. In this case,
the odd order components of $A_{1}(q)$ are $(q-1)/2$ and $q$. If
$(q-1)/2=2^p-1$, then $2^{p+1}-r^f=1$. Noting Lemma \ref{l7}, we
deduce that $f=1$, and hence $r=2^{p+1}-1$ is a Mersenne prime,
which is a contradiction because $p+1$ is not a prime. The case
when $q=2^p-1$ is similar to the previous paragraph.
\begin{step}
$K/H\ncong A_r(q)$, where $(q-1)\big |(r+1)$.
\end{step}
Recall that
$$|K/H|=|A_r(q)|=\frac{1}{(r+1, q-1)}q^{r(r+1)/2}\prod\limits_{i=2}^{r+1}(q^i-1),$$
and the only odd order component of $A_r(q)$ is $(q^r-1)/(q-1)$,
and so $(q^r-1)/(q-1)=2^p-1$. As a simple observation we see that
$q^r-1\geqslant (q^r-1)/(q-1)=2^p-1$ and so $q^r\geqslant 2^p$.
Let $q=t^f$, where $t$ is a prime number and $f$ is a natural
number.
\begin{itemize}
\item[$(i)$]
Suppose first that $r\geq 7$. Then
$q^{\frac{r(r+1)}{2}}>q^{3(r+1)}\geqslant 2^3q^{3r} \geqslant
2^{3(p+1)}$. Now, if $t$ is an odd prime, then by Lemma \ref{l8}
we obtain
$$q^{r(r+1)/2}=|K/H|_t\leqslant |G|_t< 2^{3p},$$
which is a contradiction. Therefore, we may assume that $t=2$. In
this case, we have $(2^{fr}-1)/(2^f-1)=2^p-1$, from which one can
deduce that $f=1$ and $r=p$. Thus
$$|G/K|\cdot |H|=\frac{2^{p^2}\prod_{i=1}^{p}(2^{2i}-1)}{2^{\frac{p(p+1)}{2}}\prod_{i=2}^{p+1}(2^i-1)}.$$
Since $|G/K|$ divides $|{\rm Out}(K/H)|= |{\rm Out}(A_p(2))|=2$,
we conclude that $|H|$ is divisible by $2^p+1$. Let $s\in
 {\rm ppd}(2^{2p}-1)\subseteq \pi(2^p+1)$ and $Q\in {\rm
Syl}_s(H)$. Clearly $|Q| \big| 2^p+1$. Since $H$ is a normal
$\pi_1$-subgroup of $G$ which is nilpotent, $Q$ is also a normal
$\pi_1$-subgroup of $G$. Now, by Lemma \ref{l4}, $m_2(G)=2^p-1$
divides $|Q|-1$, and so $|Q|\geqslant 2^p$. But, this forces
$|Q|=2^p+1$. However, this contradicts the fact that $2^p-1|2^p$.
\item[$(ii)$]
Suppose next that $r=5$. If $q$ is even, then from
$(q^5-1)/(q-1)=2^p-1$, we obtain $q(q^3+q^2+q+1)=2(2^{p-1}-1)$,
which implies that $q=2$ and $r=p=5$. Therefore, by easy
calculations we see that
$$|G/K|\cdot |H|=\frac{2^{10}\prod_{i=1}^{5}(2^i+1)}{2^6-1},$$
which is not a natural number, a contradiction. If $q$ is odd,
then we get
$$q(q+1)(q^2+1)=q^4+q^3+q^2+q=2^p-2,$$
however $q(q+1)(q^2+1)\equiv 0 \pmod{4}$, while $2^p-2\equiv 2
\pmod{4}$, a contradiction.
\item[$(iii)$]
Finally suppose that $r=3$. Then $q(q+1)=2(2^{p-1}-1)$. First of
all, we note that $q$ is not even, otherwise $p=3$, an impossible.
In addition, we have
\begin{equation}\label{e2}
q(q+1)=2(2^{\frac{p-1}{2}}-1)(2^{\frac{p-1}{2}}+1).
\end{equation}
Now we consider two cases separately: \subitem{$(a)$} If $q$
divides $2^{\frac{p-1}{2}}-1$, then
$$q\leq 2^{\frac{p-1}{2}}-1, \ \ q+1<2^{\frac{p-1}{2}}+1.$$
Hence, we obtain
$$q(q+1)<2(2^{\frac{p-1}{2}}-1)(2^{\frac{p-1}{2}}+1),$$
a contradiction. \subitem{$(b)$} If $q$ divides
$2^{\frac{p-1}{2}}+1$, then $2^{\frac{p-1}{2}}+1=kq$ for some
natural number $k$. Now from Eq. (\ref{e2}), it follows that
$$2k(2^{\frac{p-1}{2}}-1)=q+1.$$
If $k=1$, then $p=q=5$. Hence $13\in \pi(K/H)=\pi(A_3(5))$ but
$13\notin \pi(G)=\pi(C_5(2))$, a contradiction. Thus $k\geq 2$
and we obtain
$$2(2^{\frac{p+1}{2}}-2)-1\leq q<q+1\leqslant kq=2^{\frac{p-1}{2}}+1,$$
a contradiction.
\end{itemize}
\begin{step}
$K/H\ncong A_{r-1}(q)$, where $(r, q)\neq (3, 2), (3, 4)$.
\end{step}
Again, we recall that
$$|K/H|=|A_{r-1}(q)|=\frac{1}{(r, q-1)}q^{r(r-1)/2}\prod\limits_{i=2}^{r}(q^i-1),$$
and the only odd order component of $A_{r-1}(q)$ is
$(q^r-1)/(q-1)(r, q-1)$. Hence, we must have  $$(q^r-1)/(q-1)(r,
q-1)=2^p-1,$$  which implies that $$q^r-1\geqslant
(q^r-1)/(q-1)(r, q-1)=2^p-1,$$ or equivalently $q^r\geqslant
2^p$. Let $q=t^f$, where $t$ is a prime and $f$ is a natural
number. In what follows, we consider several cases separately.
\begin{itemize}
\item[$(i)$] $r\geq 7$. In this case, we obtain
$$q^{r(r-1)/2}\geqslant q^{3r}\geqslant 2^{3p},$$
and Lemma \ref{l8} implies that $t=2$. Now, Lemma \ref{zsig} shows
that $q=2$ and $r=p$, and hence we obtain
$$|G/K|\cdot |H|=\frac{2^{p^2}\prod_{i=1}^{p}(2^{2i}-1)}
{2^{\binom{p}{2}}\prod_{i=2}^{p}(2^i-1)}=
2^{\frac{p(p+1)}{2}\prod_{i=1}^{p}(2^i+1)}.$$ On the other hand,
$|G/K|$ divides $|{\rm Out}(K/H)|=2$. From this we deduce that
$|H|$ is divisible by $2^p+1$. Let $s\in {\rm
ppd}(2^{2p}-1)\subseteq \pi(2^p+1)$ and $Q\in {\rm Syl}_s(H)$.
Evidently  $Q$ is a normal subgroup of $G$ and $|Q|$ divides
$2^p+1$. Now, it follows from Lemma \ref{l4} that
$m_2(G)=2^p-1\big | |Q|-1$, which is impossible.

\item[$(ii)$] $r=5$. Assume first that $(5, q-1)=1$. In this case, we have
$$\frac{q^5-1}{q-1}=q^4+q^3+q^2+q+1=2^p-1,$$
or equivalently
\begin{equation}\label{equation-3}
q(q+1)(q^2+1)=2(2^{p-1}-1).
\end{equation}
If $q$ is even, then we conclude that $q=2$ and $r=p=5$, and the
proof is quite similar as $(i)$. If $q$ is odd, then the
left-hand side of Eq. (\ref{equation-3}) is congruent to $0
\pmod{4}$, while the right-hand side of Eq. (\ref{equation-3}) is
congruent to $2 \pmod{4}$, a contradiction.

Assume next that $(5, q-1)=5$. In this case, we have
$$q^4+q^3+q^2+q+1=5(2^p-1),$$ or equivalently
$$
(q-1)(q^3+2q^2+3q+4)=10(2^{p-1}-1).
$$
In the case when $q$ is even, it follows that $q=2$ and so $13=5(
2^{p-1}-1)$, a contradiction. Moreover, if $q$ is odd, then from
the equality $q(q+1)(q^2+1)=5\cdot 2^p-6$ it is easily seen that
the left-hand side of this equation is congruent to $0 \pmod{4}$,
while the right-hand side is congruent to $2 \pmod{4}$, a
contradiction.
\item[$(iii)$]
$r=3$. In this case, we have $(q^3-1)/(q-1)(3, q-1)=2^p-1$. First
of all, if $q$ is even, then we obtain $p=3$, which is not the
case. Thus, we can assume that $q$ is odd.

If $(3, q-1)=1$, then
\begin{equation}\label{equation-4}
q(q+1)=2(2^{\frac{p-1}{2}}-1)(2^{\frac{p-1}{2}}+1).
\end{equation}
If $q$ divides $2^{\frac{p-1}{2}}-1$, then
$$q\leq 2^{\frac{p-1}{2}}-1, \ \ q+1<2^{\frac{p-1}{2}}+1.$$
Hence, we obtain
$$q(q+1)<2(2^{\frac{p-1}{2}}-1)(2^{\frac{p-1}{2}}+1),$$
a contradiction. If $q$ divides $2^{\frac{p-1}{2}}+1$, then
$2^{\frac{p-1}{2}}+1=kq$. Now, from Eq. (\ref{equation-4}), it
follows that
$$2k(2^{\frac{p-1}{2}}-1)=q+1.$$
When $k=1$, we conclude that $p=5$ and $q=5$. But then, we have
$|K/H|=|A_2(5)|=2^5\cdot 3\cdot 5^3\cdot 31$, while
$|G|=|C_5(2)|=2^{25}\cdot 3^6\cdot 5^2\cdot 7\cdot 11\cdot
17\cdot 31$, this is a contradiction because $|K/H|_5>|G|_5$. If
$k\geq 2$, then $q\geq 2(2^{\frac{p+1}{2}}-2)-1$. Therefore, we
have
$$2(2^{\frac{p+1}{2}}-2)-1 \leq q < q+1 \leq 2^{\frac{p-1}{2}}+1,$$
a contradiction.

If $(3, q-1)=3$, then $q(q+1)=2^2(3\cdot 2^{p-2}-1)$, which
implies that $4|| q+1$ and so $2|| q-1$. Moreover, under these
conditions, one can easily deduce that $f$ is odd, otherwise
$8|q-1=t^f-1=(t^{\frac{f}{2}}-1)(t^{\frac{f}{2}}+1)$, which is a
contradiction. Thus, we have $|A_2(q)|_2=2^4$, while
$$|G/K|_2\cdot |H|_2=\frac{|G|_2}{|A_2(q)|_2}=2^{p^2-4}.$$
Since $|G/K|$ divides $2f(3, q-1)$ and $f$ is odd,
$|G/K|_2\leqslant 2$. Therefore a Sylow $2$-subgroup of $H$ has
order either $2^{p^2-4}$ or $2^{p^2-5}$. Applying Lemma \ref{l4}
we deduce that $2^p-1|2^{p^2-4}-1$ or $2^p-1|2^{p^2-5}-1$. Now,
one can easily check that the second divisibility is possible
only for $p=5$. But then, we get $q(q+1)=2^2\cdot 23$, which is a
contradiction.
\end{itemize}
\begin{step}
$K/H\ncong {^2A}_r(q)$, where $(q+1)\big |(r+1)$ and $(r, q)\neq
(3, 3)$, $(5, 2)$.
\end{step}
In this case, we have
$$|K/H|=|{^2A}_r(q)|=\frac{1}{(r+1,q+1)}q^{r(r+1)/2}\prod\limits^{r+1}_{i=2}\big(q^i-(-1)^i\big),$$ and the odd order
component of ${^2A}_r(q)$ is $(q^r+1)/(q+1)$, and so
$(q^r+1)/(q+1)=2^p-1$. An argument similar to that in the
previous cases shows that $q^r-1>(q^r+1)/(q+1)=2^p-1$, and so
$q^r>2^p$. Let $q=t^f$. We now consider three cases separately.
\begin{itemize}
\item[$(i)$]
$r\geq 7$. Then $q^{\frac{r(r+1)}{2}}>q^{3(r+1)}\geqslant
2^3q^{3r}>2^{3(p+1)}$, which forces by Lemma \ref{l8} that $t=2$.
Thus $(2^{fr}+1)/(2^f+1)=2^p-1$, and, consequently, $f=1$, $r=3$
and $p=2$, which is a contradiction.

\item[$(ii)$]  If $r=5$, then $(q^5+1)/(q+1)=2^p-1$.
Arguing as in the case $(i)$, we conclude that $t=2$ and $f=1$,
whence $12=2^p$, a contradiction.

\item[$(iii)$] If $r=3$, then $(q^3+1)/(q+1)=2^p-1$. It follows that
$q(q-1)=2(2^{p-1}-1)$, and so $q=p=2$, which is impossible.
\end{itemize}
\begin{step}
$K/H\ncong {^2A}_{r-1}(q)$.
\end{step}
In this case, we have
$$|K/H|=|{^2A}_{r-1}(q)|=\frac{1}{(r, q+1)}q^{r{(r-1)}/2}\prod\limits^{r}_{i=2}\big(q^i-(-1)^i\big),$$ and
the odd order component of ${^2A}_{r-1}(q)$ is $(q^r+1)/(q+1)(r,
q+1)$. Thus $$\frac{q^r+1}{(q+1)(r, q+1)}=2^p-1,$$ As before, we
deduce that $q^r\geq 2^{p}$. Let $q=t^f$. We now consider three
cases separately.
\begin{itemize}
\item[$(i)$]
$r\geqslant 7$. It follows that $q^{r(r-1)/2}\geqslant
q^{3r}>2^{3p}$, which implies that $t=2$ by Lemma \ref{l8}. Now,
we obtain $\frac{2^{fr}+1}{(2^f+1)(r, 2^f+1)}=2^p-1$, which
contradicts Lemma \ref{zsig} because $2^p-1$ is the largest prime
in $\pi(G)$.
\item[$(ii)$]
$r=5$. In this case we have $q^5+1=(q+1)(2^p-1)(5, q+1)$.

Assume first that $q$ is even, that is $q=2^f$. If $(5, q+1)=1$,
then we obtain $2^{5f}=2^{fp}+2^p-2^f-2$, which is impossible. If
$(5, q+1)=5$, then $2^{5f}=5(2^{fp}+2^p-2^f)-6$, which is again a
contradiction.

Assume next that $q$ is odd. Noting that $q(q-1)(q^2+1)=(2^p-1)(5,
q+1)-1$, it is easily seen that the left hand side is congruent to
$0 \pmod{4}$, while the right hand side is congruent to $2
\pmod{4}$, a contradiction.
\item[$(iii)$]
$r=3$. In this case, we have $(q^3+1)/(q+1)(3, q+1)=2^p-1$. If
$(3, q+1)=1$, then we obtain
$$q(q-1)=2^p-2= 2(2^{\frac{p-1}{2}}-1)(2^{\frac{p-1}{2}}+1).$$
If $q$ divides $2$, than $p=2$, a contradiction. If $q$ divides
$2^{\frac{p-1}{2}}-1$ or $2^{\frac{p-1}{2}}+1$, then
$$q(q-1)<2^p-2= 2(2^{\frac{p-1}{2}}-1)(2^{\frac{p-1}{2}}+1),$$
a contradiction. Therefore we may assume that $(3, q+1)=3$. If
$q$ is even, then we conclude that $q=4$, which is a
contradiction. We now suppose that $q$ is odd. Since
$q(q-1)=2^2(3\cdot 2^{p-2}-1)$, it follows that $4|| q-1$, and so
$2|| q+1$. Moreover, under these hypotheses, one can easily deduce
that $f$ is odd, otherwise
$8|q-1=t^f-1=(t^{\frac{f}{2}}-1)(t^{\frac{f}{2}}+1)$, which is a
contradiction. On the other hand, $|G/K|$ divides $f(3, q+1)$ and
since $f$ is odd, $|G/K|_2=1$. Therefore a Sylow $2$-subgroup of
$H$ has order $2^{p^2-4}$. Again, using Lemma \ref{l4}, we see
that $2^p-1| 2^{p^2-4}-1$, which implies that $p=2$. This is a
contradiction.
\end{itemize}
\begin{step}
$K/H\ncong C_r(2)$.
\end{step}
The odd order component of $C_r(2)$ is $2^r-1$. Thus
$2^r-1=2^{p}-1$. It follows that $r=p$, $G/K=1$ and $H=1$, which
means $G\cong C$. This completes the proof of the theorem. $\Box$

\section{Appendix}
In a series of articles, it was shown that many finite simple
groups are OD-characterizable or 2-fold OD-characterizable. Table
1 lists finite simple groups which are currently known to be
$k$-fold OD-characterizable for $k\in \{1, 2\}$. Until recently,
no examples of simple groups $P$ with $h_{\rm OD}(P)\geqslant 3$
were known. Therefore, we posed the following question:
\begin{problem} Is there a non-abelian simple group $P$ with $h_{\rm
OD}(P)\geqslant 3$?
\end{problem}

\begin{center}
{\bf Table 1}. Some non-abelian simple groups $S$ with $h_{\rm
OD}(S)=1$ or $2$.\\[0.4cm]
$\begin{array}{l|l|c|l} \hline S & {\rm Conditions \ on} \ S&
h_{\rm OD}(S) & {\rm Refs.} \\ \hline
 \mathbb{A}_n & \ n=p, p+1, p+2 \ (p \ {\rm a \ prime})& 1 &  \cite{mz1}, \cite{mzd}    \\
 & \ 5\leqslant n\leqslant 100, n\neq 10   & 1 & \cite{hm}, \cite{kogani}, \cite{banoo-alireza},  \\
 & & & \cite{mz3}, \cite{zs-new2} \\
& \ n=106, \ 112 & 1 &      \cite{yan-chen}       \\
& \ n=10 & 2 &      \cite{mz2}       \\[0.2cm]
L_2(q) &  q\neq 2, 3& 1 &    \cite{mz1}, \cite{mzd},\\
& & &  \cite{zshi} \\[0.1cm]
L_3(q) &  \ |\pi(\frac{q^2+q+1}{d})|=1, \ d=(3, q-1) & 1 &   \cite{mzd} \\[0.2cm]
U_3(q) &  \ |\pi(\frac{q^2-q+1}{d})|=1, \ d=(3, q+1), q>5 & 1 &   \cite{mzd} \\[0.2cm]
L_3(9) & & 1 & \cite{zs-new4}\\[0.1cm]
U_3(5) &   & 1 &   \cite{zs-new5} \\[0.1cm]
L_4(q) &  \ q\leqslant 17  & 1 &   \cite{bakbari, amr} \\[0.1cm]
U_4(7) &   & 1 &   \cite{amr} \\[0.1cm]
L_n(2) & \ n=p \ {\rm or} \ p+1, \ {\rm for \ which} \ 2^p-1 \ {\rm is \ a \ prime} & 1 & \cite{amr} \\[0.2cm]
L_n(2) & \ n=9, 10, 11  & 1 &   \cite{khoshravi}, \cite{R-M} \\[0.1cm]
U_6(2) & & 1 & \cite{LShi} \\[0.1cm]

R(q) & \ |\pi(q\pm \sqrt{3q}+1)|=1, \ q=3^{2m+1}, \ m\geqslant 1 & 1 & \cite{mzd} \\[0.2cm]
{\rm Sz} (q) & \ q=2^{2n+1}\geqslant 8& 1 &   \cite{mz1}, \cite{mzd} \\[0.2cm]
B_m(q), C_m(q) &  m=2^f\geqslant 4,  \
|\pi\big((q^m+1)/2\big)|=1, \
 & 2 & \cite{akbarim}\\[0.2cm]
B_2(q)\cong C_2(q) &  \ |\pi\big((q^2+1)/2\big)|=1, \ q\neq
3 & 1 & \cite{akbarim}\\[0.2cm]
B_m(q)\cong C_m(q) &  m=2^f\geqslant 2, \ 2|q, \
|\pi\big(q^m+1\big)|=1, \ (m, q)\neq (2, 2) & 1 &
\cite{akbarim}\\[0.2cm]
B_p(3), C_p(3) &  |\pi\big((3^p-1)/2\big)|=1, \  p \ {\rm is \ an
\ odd \
prime}  & 2 & \cite{akbarim}, \cite{mzd}\\[0.2cm]
B_3(5), C_3(5) & & 2 & \cite{akbarim} \\[0.2cm]
C_3(4) & & 1 & \cite{moghadam} \\[0.1cm]
S &  \ \mbox{A sporadic simple group} & 1 & \cite{mzd} \\[0.1cm]
S &  \ \mbox{A simple group with} \ |\pi(S)|=4, \ \ S\neq \mathbb{A}_{10} & 1 & \cite{zs} \\[0.1cm]
S &  \  \mbox{A simple group with} \ |S|\leqslant 10^8, \ \ S\neq \mathbb{A}_{10}, \ U_4(2) & 1 & \cite{ls} \\[0.1cm]
S &  \  \mbox{A simple $C_{2,2}$- group} & 1 & \cite{mz1}
\end{array}$
\end{center}
\footnotetext{In Table 3, $q$ is a power of a prime number.}

Although we have not found a simple group which is $k$-fold
OD-characterizable for $k\geqslant 3$, but among non-simple
groups, there are many groups which are $k$-fold
OD-characterizable for $k\geqslant 3$. As an easy example, if $P$
is a $p$-group of order $p^n$, then $h_{\rm OD}(P)=\nu(p^n)$,
where $\nu(m)$ signifies the number of non-isomorphic groups of
order $m$. Table 2 lists finite non-solvable groups which are
currently known to be OD-characterizable or $k$-fold
OD-characterizable with $k\geqslant 2$.

\begin{center}
{\bf Table 2}. Some non-solvable groups $G$ with known $h_{\rm
OD}(G)$.\\[0.2cm]
$\begin{array}{l|l|c|l} \hline G & {\rm Conditions \ on} \ G &
h_{\rm OD}(G) & {\rm Refs.} \\ \hline
{\rm Aut}(M) & M \ \mbox{is a sporadic group}  \neq  J_2, M^cL    & 1 & \cite{mz1} \\[0.1cm]
\mathbb{S}_n & n=p, \ p+1 \ (p\geqslant 5 \ \mbox{is a prime})& 1 &  \cite{mz1}    \\[0.1cm]
U_3(5): 2 & & 1 & \cite{zs-new5} \\[0.1cm]
U_6(2): 2 & &1& \cite{LShi}\\[0.1cm]
M &  M\in \mathcal{C}_1 & 2 &      \cite{mz2}       \\[0.1cm]
M & M\in \mathcal{C}_2 & 8 &      \cite{mz2}       \\[0.1cm]
M & M\in \mathcal{C}_3  & 3 & \cite{hm, kogani, banoo-alireza, mz3, yan-chen} \\[0.1cm]
M & M\in \mathcal{C}_4     & 2 & \cite{mz2} \\[0.1cm]
M & M\in \mathcal{C}_5    & 3 & \cite{mz2} \\[0.1cm]
M & M\in \mathcal{C}_6  & 6 &\cite{banoo-alireza} \\[0.1cm]
M & M\in \mathcal{C}_7 & 1 &  \cite{zs-new1} \\[0.1cm]
M & M\in \mathcal{C}_8  & 9 &\cite{zs-new1} \\[0.1cm]
M & M\in \mathcal{C}_{9} & 3 &  \cite{zs-new5} \\[0.1cm]
M & M\in \mathcal{C}_{10}  & 6 &\cite{zs-new5}\\[0.1cm]
M & M\in \mathcal{C}_{11}  & 3 &\cite{LShi}\\[0.1cm]
M & M\in \mathcal{C}_{12}  & 5 &\cite{LShi}\\[0.1cm]
M & M\in \mathcal{C}_{13}  & 1 &\cite{y-chen-w}\\[0.1cm]
M & M\in \mathcal{C}_{14}  & 1 &\cite{R-M}
\end{array}$
\end{center}
\begin{tabular}{lll}
$\mathcal{C}_1$ & $\!\!\!\!=$  & $ \!\!\!\! \{\mathbb{A}_{10}, J_2\times \mathbb{Z}_{3} \}$\\[0.1cm]
$\mathcal{C}_2$ & $\!\!\!\!=$ & $\!\!\!\! \{ \mathbb{S}_{10},  \
\mathbb{Z}_{2}\times \mathbb{A}_{10}, \ \mathbb{Z}_2\cdot
\mathbb{A}_{10}, \ \mathbb{Z}_6\times J_2,  \ \mathbb{S}_3 \times
J_2, \ \mathbb{Z}_3\times (\mathbb{Z}_2\cdot J_2),
$ \\[0.1cm]
& & $(\mathbb{Z}_3\times J_2)\cdot \mathbb{Z}_2, \ \mathbb{Z}_3\times {\rm Aut}(J_2)\}.$\\[0.1cm]
$ \mathcal{C}_3$ & $\!\!\!\!=$ & $\!\!\!\! \{\mathbb{S}_{n}, \
\mathbb{Z}_{2}\cdot \mathbb{A}_{n}, \ \mathbb{Z}_{2}\times
\mathbb{A}_{n}\}$, \
where $9\leqslant n\leqslant 100$ with $n\neq 10, p, p+1$ ($p$ a prime)\\[0.1cm]
& & or $n=106, \ 112$.\\[0.1cm]
$\mathcal{C}_4$ & $\!\!\!\!=$ & $\!\!\!\! \{ {\rm Aut}(M^cL), \ \mathbb{Z}_2\times M^cL\}$.\\[0.1cm]
$\mathcal{C}_5$ & $\!\!\!\!=$ & $\!\!\!\! \{{\rm Aut}(J_2), \
\mathbb{Z}_2\times J_2, \
\mathbb{Z}_2\cdot J_2\}.$\\[0.1cm]
$\mathcal{C}_6$ & $\!\!\!\!=$ & $\!\!\!\! \{{\rm Aut}(S_6(3)), \
\mathbb{Z}_2\times S_6(3), \  \mathbb{Z}_2\cdot S_6(3),  \
\mathbb{Z}_2\times O_7(3), \
\mathbb{Z}_2\cdot O_7(3)$, \ ${\rm Aut}(O_7(3))\}$.\\[0.1cm]
$\mathcal{C}_7$ & $\!\!\!\!=$ & $\!\!\!\! \{L_2(49):2_1, \
L_2(49):2_2, \
L_2(49):2_3\}$.\\[0.1cm]
$\mathcal{C}_8$ & $\!\!\!\!=$ & $\!\!\!\! \{L\cdot 2^2,  \
\mathbb{Z}_{2}\times (L: 2_1), \ \mathbb{Z}_2\times (L: 2_2), \
\mathbb{Z}_2\times (L: 2_3),  \ \mathbb{Z}_2\cdot(L: 2_1), $
\\[0.1cm]
& & $\mathbb{Z}_2\cdot(L: 2_2), \ \mathbb{Z}_2\cdot(L: 2_3), \
\mathbb{Z}_4\times L, \ (\mathbb{Z}_2\times \mathbb{Z}_2)\times
L\}$, \ where $L=L_2(49)$.\\[0.1cm]
$\mathcal{C}_{9}$ & $\!\!\!\!=$  & $ \!\!\!\! \{U_3(5):3, \
\mathbb{Z}_{3} \times U_3(5),
\ \mathbb{Z}_{3}\cdot U_3(5) \}$\\[0.1cm]
$\mathcal{C}_{10}$ & $\!\!\!\!=$  & $ \!\!\!\! \{L:\mathbb{S}_3, \
\mathbb{Z}_{2}\cdot(L:3), \ \mathbb{Z}_{3}\times (L:2), \
\mathbb{Z}_{3}\cdot(L:2), \ (\mathbb{Z}_{2}\times L): 2$, \
$(\mathbb{Z}_{3}\cdot L): 2\},$
\\[0.1cm]
& & where
$L=U_3(5)$.\\[0.1cm]
$\mathcal{C}_{11}$ & $\!\!\!\!=$  & $ \!\!\!\! \{U_6(2): 3, \
\mathbb{Z}_{3}\times U_6(2), \ \mathbb{Z}_{3}\cdot U_6(2)\}$.\\[0.1cm]
$\mathcal{C}_{12}$ & $\!\!\!\!=$  & $ \!\!\!\! \{L\cdot
\mathbb{S}_3, \ \mathbb{Z}_{3}\times (L:2), \ \mathbb{Z}_{3}\cdot
(L:2), \ (\mathbb{Z}_{3}\times L):2, \ (\mathbb{Z}_{3}\cdot
L):2\}$, where
$L=U_6(2)$.\\[0.1cm]
$\mathcal{C}_{13}$ & $\!\!\!\!=$  & $ \!\!\!\! \{{\rm
Aut}(O^+_{10}(2), \  {\rm Aut}(O^-_{10}(2)\}$,
\\[0.1cm]
$\mathcal{C}_{14}$ & $\!\!\!\!=$  & $ \!\!\!\! \{{\rm
Aut}(L_{p}(2)), \ {\rm Aut}(L_{p+1}(2))\}$, where $2^p-1$ is a
Mersenne prime.
\\[0.1cm]
\end{tabular}

\end{document}